\documentclass[11pt,a4paper]{article}

\usepackage{amscd,amssymb,amsthm}
\newtheorem{thm}{Theorem}

\newtheorem{defi}{Definition}

\title{Stable ample 2-vector bundles on Hirzebruch surfaces}
\author{by Alexandru Sterian}
\date{}

\usepackage{amsfonts, amssymb, amscd}
\usepackage{verbatim}
\usepackage{eucal}
\usepackage{amssymb}
\usepackage{mathrsfs}
\usepackage{graphicx}

\addtocounter{footnote}{1}

\theoremstyle{definition}

\theoremstyle{remark}
\newtheorem*{proof*}{Proof}

\begin{document}
\maketitle
\begin{abstract}
We discuss stability conditions for all rank-2 ample vector
bundles on Hirzebruch surfaces with the second Chern class less
than 7.
\end{abstract}

\begin{center}\textbf{AMS Mathematics Subject Classification:} Primary 14J60; Secondary 14D20,
14J26.\end{center}
\begin{center} Stable ample-2 vector
bundles\end{center}
\begin{center}\textbf{Key Words:}Rank 2-vector bundles,
Hirzebruch surface, stability conditions, Chern classes.
\end{center}

\section{Introduction}
Ample vector bundles with small Chern numbers on surfaces have
been studied by several authors; see \cite{BL}, \cite{LS},
\cite{N} and \cite{Hi}. Ishihara studied ample vector bundles of
rank 2 on a Hirzebruch surface. On the other hand, the notion of
stability has proved to be crucial in the construction of moduli
spaces of vector bundles with given numerical invariants.

In this paper, we discuss the stability of rank-2 ample vector
bundles on Hirzebruch surfaces with the second Chern class less
than 7. We mainly rely on a classification theorem obtained by
Ishihara \cite{Hi} and a stability criterion proved by Aprodu and
Br\^inz\u anescu \cite{A-Br 1}.
\newpage
\section{Hirzebruch surfaces}

We shall use classical notations and facts on
Hirzebruch surfaces, and we refer to \cite{Ha} or \cite{A-Br 1} for
more details. 
\newline Let $X=\mathbb{P}(\mathcal{O}_{\mathbb{P}^1}\oplus\mathcal{O}_{\mathbb{P}^1}(-e))\stackrel{\pi}{\rightarrow}\mathbb{P}^1$ be a Hirzebruch surface with invariant $e\geq 0$. Let $C_0$ be a section of $X$ with $\mathcal{O}_X(C_0)\cong\mathcal{O}_{\mathbb{P}(\mathcal{E})}(1)$, where $\mathcal{E}=\mathcal{O}_{\mathbb{P}^{1}}\oplus\mathcal{O}_{\mathbb{P}^1}(-e)$ and let $f_0$ be a fixed fibre. Then ${C_0}^2=-e$, $C_0f_0=1$ and ${f_0}^2=0$.
\newline Let $E$ be a rank-2 algebraic vector bundle on $X$ with fixed numerical Chern classes $c_1=(\alpha,\beta)\in H^2(X,\mathbb{Z})\cong\mathbb{Z}^2$ and $c_2=\gamma\in H^4(X,\mathbb{Z})\cong\mathbb{Z}$, where $\alpha, \beta, \gamma\in\mathbb{Z}$. Since the fibres of $\pi$ are isomorphic to $\mathbb{P}^1$, we can speak about the generic splitting type of $E$; for a general fibre $f$ of $\pi$, we have
$$E|_{f}\cong\mathcal{O}_f(d)\oplus\mathcal{O}_f(d'),$$
with $d'\leq d$ and $d+d'=\alpha$. We will call $d$ the first numerical invariant of $E$, and we define the second invariant of $E$ (see \cite{Br}) to be:
$$-r=inf\{\,l\,| \,\exists\,L\in Pic(\mathbb{P}^1),\, \deg{L}=l,\,\, s.t.\,\, H^0(X, E(-dC_0)\oplus\pi^*L)\neq 0\}.$$
Since $Pic_0(\mathbb{P}^1)$ is trivial, the following result (see \cite{Br}) takes the form:

\begin{thm}\label{constructie}(see \cite{Br})
For every $E$ rank-2 vector bundle on a Hirzebruch surface $X$, with
fixed Chern classes $c_1=(\alpha,
\beta)\in\mathbb{Z}\times\mathbb{Z},\,\,c_2\in\mathbb{Z}$ and
invariants $d$ and $r$, there exist
$Y\subset X$ a locally complete intersection of codimension 2 in
$X$, or the empty set, such that $E$ is given by the extension:

$$
0\rightarrow\mathcal{O}_X(dC_0+rf_0)\rightarrow
E\rightarrow\mathcal{O}_X(d'C_0+sf_0)\otimes
I_Y\rightarrow 0,\,\,\,(\bigstar)$$
where $d+d'=\alpha$, $r+s=\beta$ and
$\deg(Y)=c_2+\alpha(de-r)-\beta d+2dr-d^2e\geq 0$.
\end{thm}

\section{Ampleness on Hirzebruch surfaces}

The aim of this section is to present some generalities of ample
vector bundles on Hirzebruch surface (see \cite{Ha} and
\cite{Hi}).

Given a divisor $D=aC_0+bf_0$ on a Hirzebruch surface, it is well
known that $D$ is ample if and only if it is very ample, if and
only if $a>0$ and $b>ae$. We also know that a divisor
$D=a'C_0+b'f_0$ is effective if and only if $a'\geq 0$ and $b'\geq
0$.
\begin{defi}
Let $E$ be a vector bundle on $X$,
$\mathbb{P}(E)$ the associated projective bundle and
$H_{E}$ the tautological line bundle on
$\mathbb{P}(E)$. We say that $E$ is ample if
$H_{E}$ is ample.
\end{defi}

Ample rank-2 vector bundles on Hirzebruch surfaces
$X=\mathbb{P}(\mathcal{O}_{\mathbb{P}^1}\oplus\mathcal{O}_{\mathbb{P}^1}(-e))$
were classified  by Ishihara in \cite{Hi}. Since the decomposable bundles are not stable, we express the results of Ishihara only for indecomposables bundles:

\begin{thm}\label{caracterizare} (see \cite{Hi})
Let $X=\mathbb{P}(\mathcal{O}_{\mathbb{P}^1}\oplus\mathcal{O}_{\mathbb{P}^1}(-e))$ be a Hirzebruch surface and let $E$ be an indecomposable ample vector bundle of rank-2 on $X$. Then $c_2(E)\geq e+5$. Moreover, if $c_2(E)\leq e+6$ then one of the following cases occurs: 
\begin{enumerate}
\item $e=1$, $c_2=6$  and $E$ is given by the non-trivial extension 
$\newline 0\rightarrow \mathcal{O}(2C_0+2f_0)\rightarrow
E\rightarrow\mathcal{O}(C_0+3f_0)\rightarrow 0;$
\item $e=0$, $c_2=6$ and $E$ is given by the non-trivial extension
$\newline 0\rightarrow \mathcal{O}(2C_0)\rightarrow
E\rightarrow\mathcal{O}(C_0+3f_0)\rightarrow 0;$ 
\item $e=1$, $c_2=7$ and $E$ is given by the non-trivial extension
$\newline0\rightarrow \mathcal{O}(2C_0+f_0)\rightarrow
E\rightarrow\mathcal{O}(C_0+4f_0)\rightarrow 0;$
\item $e=2$, $c_2=8$ and $E$ is given by the non-trivial extension 
$\newline0\rightarrow \mathcal{O}(2C_0+4f_0)\rightarrow
E\rightarrow\mathcal{O}(C_0+4f_0)\rightarrow 0.$ 
\end{enumerate}
\end{thm}

\section{Stability}
In this section we recall some facts concerning the stability in
the sense of Mumford-Takemoto of vector bundles on the Hirzebruch
surfaces $X$. For more details we refer to \cite{A-Br 1},
\cite{MR} and \cite{Q2}.
\begin{defi}
Let $H$ be an ample line bundle on a smooth projective surface
$X$. For a torsion free sheaf $E$ on $X$ we set:
$$\mu(E)=\mu_H(E):=\frac{c_1(E)H}{rk(E)}.$$
The sheaf $E$ is called semistable with respect to $H$ if
$$\mu_H(G)\leq\mu_H(E)$$ for all non-zero subsheaves $G\subset E$
with $rk(G)< rk(E)$. If strict inequality holds, then $E$ is said
to be stable with respect to $H$.
\end{defi}
We shall also use the description of Qin (see \cite{Q1} or \cite{Fr}) for walls
and chambers.
\newline A very useful criterion for stability in the case of rank-2 vector bundles over ruled surfaces was found by Aprodu and Br\^inzanescu in \cite{A-Br 1}. We remind now this criterion in the case of Hirzebruch surfaces,
which will be used later.
\begin{thm}\label{Aprodu}(see \cite{A-Br 1})
Let $E$ be a rank two vector bundle over a Hirzebruch surface, with
numerical invariants as in Theorem \ref{constructie}. Then, there
exists an ample line bundle $H$ such that $E$ is $H$ stable if and
only if $2r<\beta$ and the extension
($\bigstar$) of $E$ is non-splitting.
\end{thm}

\section{Main Result}
Using the tools presented in the above sections we are now able to
state and prove our main result:
\begin{thm}\label{prima}
Let
$X\cong\mathbb{P}(\mathcal{O}_{\mathbb{P}^1}\oplus\mathcal{O}_{\mathbb{P}^1}(-e))$
be a Hirzebruch surface. Then, the only rank-2 ample  stable
vector bundles on $X$ with $c_2(E)\leq 6$ or $c_2(E)\leq 7$ and
$e\geq 1$ are given by the following exact sequences:
\begin{enumerate}
\item $0\rightarrow \mathcal{O}(2C_0+2f_0)\rightarrow
E\rightarrow\mathcal{O}(C_0+3f_0)\rightarrow 0$ for $e=1,$
$(c_2(E)=6)$;

\item $0\rightarrow \mathcal{O}(2C_0)\rightarrow
E\rightarrow\mathcal{O}(C_0+3f_0)\rightarrow 0$ for $e=0$
$(c_2(E)=6)$.

\item $0\rightarrow \mathcal{O}(2C_0+f_0)\rightarrow
E\rightarrow\mathcal{O}(C_0+4f_0)\rightarrow 0$ for $e=1$
$(c_2(E)=7)$.

\end{enumerate}
\end{thm}
\begin{proof} According to Theorem \ref{caracterizare} (or see \cite{Hi}), there are four cases to discuss.
\newline Following the idea of the proof of Remark 1 in \cite{Br}, it is easy to see that all the four extensions from Theorem \ref{caracterizare} coincide with the extension ($\bigstar$) of the corresponding 2-vector bundles $E$'s. We show this fact only for the first case, the arguments used in the proofs of other cases being quite similar.
\newline Case A. The ample rank-2 vector bundles $E$
constructed from the extension 
\begin{equation}\label{eq1}
0\rightarrow
\mathcal{O}(2C_0+2f_0)\rightarrow
E\rightarrow\mathcal{O}(C_0+3f_0)\rightarrow 0,
\end{equation} are
non-splitting, since $Ext^1(\mathcal{O}(C_0+3f_0),
\mathcal{O}(2C_0+2f_0))\neq 0$. In this case, we know that $e=1$,
$c_1(E)=3C_0+5f_0$ and $c_2(E)=6$. Obviously, $\alpha=3, \beta=5$ and $\gamma=6$. By restricting the exact sequence (\ref{eq1}) to a general fibre $f$ we obtain
$$0\longrightarrow\mathcal{O}_{f}(2)\longrightarrow E|_f\longrightarrow\mathcal{O}_f(1)\longrightarrow 0,$$
and since $H^1(f,\mathcal{O}_f(1))=0$, it follows that $E|_f=\mathcal{O}_f(2)\oplus\mathcal{O}_f(1)$, and so using the notations from
Theorem \ref{constructie}, we get $d=2$ and $d'=1$.
\newline Considering the equivalent form of (\ref{eq1}):
\begin{equation}\label{eq2}
0\longrightarrow\mathcal{O}_X\longrightarrow E(-2C_0-2f_0)\longrightarrow\mathcal{O}_X(-C_0+f_0)\longrightarrow 0
\end{equation}
and looking to the corresponding long exact sequence of cohomology we obtain the injective map:
$$0\longrightarrow H^0(X,\mathcal{O}_X)\longrightarrow H^0(X,E(-2C_0-2f_0)).$$
Thus, $H^0(X,E(-2C_0-2f_0))\neq 0$.
To show that the second invariant  of $E$ is $r=2$, we must verify that $H^0(X,E(-2C_0-(2+m)f_0))=0$ for every $m\geq 1$.
Tensorising (\ref{eq2}) by $\mathcal{O}_X(-mf_0)$ it follows:
\newline $0\longrightarrow H^0(X,\mathcal{O}_X(-mf_0))\longrightarrow H^0(X,E(-2C_0-(2+m)f_0))\longrightarrow$
\newline $\longrightarrow H^0(X,\mathcal{O}_X(-C_0+(1-m)f_0))\longrightarrow ...\,\,.$
\newline By the projection formula:
$$H^0(X,\mathcal{O}_X(-C_0+(1-m)f_0))\cong H^0(\mathbb{P}^1,\pi_*(\mathcal{O}_X(-C_0))\otimes\mathcal{O}_{\mathbb{P}^1}(1-m)f_0))=0,$$ since
$\pi_*(\mathcal{O}_X(-C_0))=0.$
\newline Therefore, $$H^0(X,E(-2C_0-(2+m)f_0))\cong H^0(X,\mathcal{O}_X(-mf_0))\cong H^0(\mathbb{P}^1,\mathcal{O}_{\mathbb{P}^1}(-m))=0,$$ since $m\geq 1$.
We can deduce now that $r=2$ and $s=3$.
\newline Moreover, for all these invariants it follows that $\deg Y=0$, and thus $Y=\emptyset$.
\newline We conclude that $E$ sits in an extension of type $(\bigstar)$.
\newline According to Theorem \ref{Aprodu}, $E$ will be stable if and only
if will verify:$$2r<\beta, $$
 which in our case  this give
$4<5$. Hence $E$ est stable.
\newline We will proceed in the same way with the other 3 cases who left.
\newline Case B. Consider the extension:
\begin{equation}\label{eq4}0\rightarrow \mathcal{O}(2C_0)\rightarrow
E\rightarrow\mathcal{O}(C_0+3f_0)\rightarrow 0.
\end{equation} In this
case:$$e=0,\,\,c_1(E)=3C_0+3f_0,\,\,c_2(E)=6$$ and the
coefficients are:$$d=2,\,\,r=0,\,\,
d'=1,\,\,s=3,\,\,\alpha=3,\,\,\beta=3.$$ An easy computation shows
that $Ext^1(\mathcal{O}(C_0+3f_0), \mathcal{O}(2C_0))\neq 0$, so
there are non-splitt bundles given by sequence (\ref{eq4}). Since
$2r=0$ and $\beta=3$, it follows that $E$ is stable.
\newline Case C. For  $E$ gived by the non-trivial extension:
$$ 0\rightarrow \mathcal{O}(2C_0+f_0)\rightarrow
E\rightarrow\mathcal{O}(C_0+4f_0)\rightarrow 0,$$
we have:
$$e=1,\,\,c_1(E)=3C_0+5f_0,\,\,c_2(E)=7,$$$$d=2,\,\,r=1,\,\,
d'=1,\,\,s=4,\,\,\alpha=3,\,\,\beta=5.$$
Again, the inequality  $2r<\beta$ holds, and so $E$ is stable.
\newline Case D. In the last case $E$ is given by the
non-splitting extension: $$0\rightarrow
\mathcal{O}(2C_0+4f_0)\rightarrow
E\rightarrow\mathcal{O}(C_0+4f_0)\rightarrow 0.$$ For this
case
$$e=2,\,\,c_1(E)=3C_0+8f_0,\,\,c_2(E)=8,$$$$d=2,\,\,r=4,\,\,d'=1,\,\,s=4,\,\,\alpha=3,\,\,\beta=8.$$
Since $2r=\beta=8$, we conclude that
$E$ is not stable.
\end{proof}
\smallskip The next result describe the polarizations for which the ample rank-2 vector bundles from the above theorem are stable.
\begin{thm}
Let
$X\cong\mathbb{P}(\mathcal{O}_{\mathbb{P}^1}\oplus\mathcal{O}_{\mathbb{P}^1}(-e))$
be a Hirzebruch surface and $H=aC_0+bf_0$ an ample line bundle.
\begin{enumerate} \item If $e=1$ and $a<b<2a$, then any ample
rank-2 vector bundle $E$ given by the non-trivial
extension:$$0\rightarrow \mathcal{O}(2C_0+2f_0)\rightarrow
E\rightarrow\mathcal{O}(C_0+3f_0)\rightarrow 0$$ is $H$ stable.
 \item If $e=0$ and $b<3a$, then any ample rank-2 vector
bundle $E$ given by the non-trivial extension $$0\rightarrow
\mathcal{O}(2C_0)\rightarrow
E\rightarrow\mathcal{O}(C_0+3f_0)\rightarrow 0$$ is $H$ stable.
\item If $e=1$ and $a<b<4a$, then any ample rank-2 vector bundle
$E$ given by the non-trivial extension
$$0\rightarrow \mathcal{O}(2C_0+f_0)\rightarrow
E\rightarrow\mathcal{O}(C_0+4f_0)\rightarrow 0;$$ is $H$ stable.
\end{enumerate}
\end{thm}
\begin{proof}According to Theorem 1.2.3
from Qin (see \cite{Q1}), we know that if $E$ is a rank-2 vector
bundle given by a non-trivial extension of the form
$$0\rightarrow\mathcal{O}_X(G)\rightarrow
E\rightarrow\mathcal{O}_X(c_1-G)\otimes I_Y\rightarrow 0,$$ and if
we consider the wall $\zeta:=2G-c_1$ then:
\newline $E$ will be $H$-stable for every ample line bundle $H$
who verify $H.\zeta<0$;
\begin{enumerate}
\item In this case $\zeta=C_0-f_0$ , so
$$H.\zeta<0\Longleftrightarrow
(aC_0+bf_0)(C_0-f_0)<0\Longleftrightarrow$$$$
a{C_0}^2-aC_0f_0+bf_0C_0-b{f_0}^2<0\Longleftrightarrow -2a+b<0,$$
$$\,\,\textrm{since}\,\,C_0^2=-1,\,\,C_0f_0=1\,\,\textrm{and}\,\,{f_0}^2=0.$$
\item For the second case $\zeta=C_0-3f_0$ and hence:
$$H.\zeta<0\Longleftrightarrow
(aC_0+bf_0)(C_0-3f_0)<0\Longleftrightarrow$$
$$a{C_0}^2-3aC_0f_0+bf_0C_0-3b{f_0}^2<0
\Longleftrightarrow
-3a+b<0,$$$$\,\,\textrm{since}\,\,C_0^2=0,\,\,C_0f_0=1\,\,\textrm{and}\,\,{f_0}^2=0.$$
\item In the third case $\zeta=C_0-3f_0$, so we get:
\newline$$H.\zeta<0\Longleftrightarrow
(aC_0+bf_0)(C_0-3f_0)<0\Longleftrightarrow$$
$$a{C_0}^2-3aC_0f_0+bf_0C_0-3b{f_0}^2<0\Longleftrightarrow
-4a+b<0,$$$$\,\,\textrm{since}\,\,C_0^2=-1,\,\,C_0f_0=1\,\,\textrm{and}\,\,{f_0}^2=0.$$
\end{enumerate}
\end{proof}
In the next paragraph we verify directly the stability for the
first case of vector bundles from Theorem 5 with respect to a
fixed polarization:
\newline\textbf{Example.}\,\,\,Let $X$ be a Hirzebruch surface with
$e=1$, and $H=2C_0+3f_0$ an ample line bundle on $X$. We will
check directly that the nonsplit ample rank-2 vector bundles $E$
on $X$ which sit in the nontrivial extension:
$$0\rightarrow \mathcal{O}(2C_0+2f_0)\rightarrow
E\rightarrow\mathcal{O}(C_0+3f_0)\rightarrow 0\,\,(\dag)$$ are
$\mu$ stable with respect to $H$. That is, for any rank 1 subbundle
$\mathcal{O}_X(D)$ of $E$ we have
$$c_1(\mathcal{O}_X(D))H<\frac{c_1(E)H}{2}=\frac{(3C_0+5f_0)(2C_0+3f_0)}{2}=\frac{13}{2}.$$
Since $E$ is given by the exact sequence ($\dag$), we have:
\begin{enumerate}
\item $\mathcal{O}_X(D)\hookrightarrow\mathcal{O}(2C_0+2f_0)$ or
\item $\mathcal{O}_X(D)\hookrightarrow\mathcal{O}(C_0+3f_0)$.
\end{enumerate}
In the first case $(2C_0+2f_0)-D\geq 0$ (i.e. effective divisor).
Since $H$ is an ample line bundle, it's clear that
$(2C_0+2f_0-D)H\geq 0$ and hence we have
$c_1(\mathcal{O}_X(D))H=DH\leq(2C_0+2f_0)H=(2C_0+2f_0)(2C_0+3f_0)=6<\frac{13}{2}=\frac{c_1(E)H}{2}$.
\newline In the second case, we consider $D=\alpha C_0+\beta f_0$.
\newline Since $\mathcal{O}_X(D)\hookrightarrow\mathcal{O}(C_0+3f_0)$,
it results that $(C_0+3f_0)-D=(1-\alpha)C_0+(3-\beta)$ is an
effective divisor, and we get $\alpha\leq 1$ and $\beta\leq 3$.
But since $E$ is given by an extension which does not split, this
implies that $(\alpha,\beta)\neq(1,3)$. \newline In conclusion, we
have: $$(\alpha C_0+\beta f_0)(2C_0+3f_0)=\alpha+
2\beta<\frac{13}{2}=\frac{c_1(E)H}{2}.$$
\textbf{Acknowledgements.} The author expresses his thanks to Vasile Br\^inz\u anescu for the subject of this note and to the referee for helpful suggestions.

\begin{tabbing}
\hspace*{8cm} \= \hspace*{2cm} \= \hspace*{3cm} \= \hspace*{4cm}
\kill

\parbox{\textwidth}{Alexandru Sterian}\\
\parbox{\textwidth}{University "Spiru Haret"}\\
\parbox{\textwidth}{Department of Mathematics}\\
\parbox{\textwidth}{Str. Ion Ghica, 13, 030045}\\
\parbox{\textwidth}{Bucharest, Romania}\\
\parbox{\textwidth}{alexandru.sterian@gmail.com}\\
\end{tabbing}

\end{document}